\documentclass[a4paper, 12pt]{amsart}
\usepackage[italian, english]{babel}

\usepackage[latin1]{inputenc}
\usepackage[T1]{fontenc}
\usepackage{lmodern}

\usepackage{amsmath}
\usepackage{amssymb}
\usepackage{amsthm}

\usepackage{hyperref}

\usepackage{enumerate}

\newcommand{\ppar}[0]{\ \par}

\newcommand{\E}{\mathbf{E}}

\renewcommand{\phi}{\varphi}
\renewcommand{\epsilon}{\varepsilon}

\newcommand{\Size}[1]{\left\lvert #1 \right\rvert}

\newcommand{\Span}[1]{\left\langle\, #1 \,\right\rangle}

\newcommand{\Set}[1]{\left\{ #1 \right\}}

\newcommand{\Hc}[0]{\mathcal{H}}
\newcommand{\Kc}[0]{\mathcal{K}}

\newcommand{\norm}[0]{\trianglelefteq}

\renewcommand{\theta}[0]{\vartheta}
\renewcommand{\phi}[0]{\varphi}

\DeclareMathOperator{\Aut}{Aut}

\DeclareMathOperator{\Hol}{Hol}

\newtheorem{dummy}{Dummy}
\numberwithin{dummy}{section}

\newtheorem{theorem}[dummy]{Theorem}

\newtheorem{remark}[dummy]{Remark}

\newtheorem{lemma}[dummy]{Lemma}
\newtheorem{definition}[dummy]{Definition}

\newtheorem{proposition}[dummy]{Proposition}

\newtheorem*{question}{Question}

\newtheorem{rec}[dummy]{Lemma}

\newtheorem{ass}[dummy]{Assumption}

\newcommand{\Fact}[0]{Lemma}

\numberwithin{equation}{section}

\begin{document}

\date{7 March 2017, 10:34 CET --- Version 4.09%%%
}

\title[]%
      {The multiple holomorph of\\
      a finitely generated abelian group}
      
\author{A. Caranti}

\address[A.~Caranti]%
 {Dipartimento di Matematica\\
  Universit\`a degli Studi di Trento\\
  via Sommarive 14\\
  I-38123 Trento\\
  Italy} 

\email{andrea.caranti@unitn.it} 

\urladdr{http://science.unitn.it/$\sim$caranti/}

\author{F.~Dalla Volta}

\address[F.~Dalla Volta]{Dipartimento di Matematica e Applicazioni\\
  Edificio U5\\
  Universit\`a degli Studi di Milano--Bicocca\\
  Via R.~Cozzi, 53\\
  I-20126 Milano\\
  Italy}

\email{francesca.dallavolta@unimib.it}

\urladdr{http://www.matapp.unimib.it/$\sim$dallavolta/}

\subjclass[2010]%
                {%              
                20B35 %Permutation groups/Subgroups of symmetric groups
                20K05 %Abelian groups/Finitely generated groups
                20K30 %Abelian groups/Automorphisms, homomorphisms,
                %endomorphisms, etc.
                20D45 %
                20E36%
                }

\keywords{holomorph, multiple holomorph, finitely-generated abelian
  groups, commutative rings}

\begin{abstract}        
  W.H.~Mills   has   determined,  for   a   finitely
  generated abelian  group $G$, the  regular subgroups $N \cong  G$ of
  $S(G)$, the  group of permutations  on the  set $G$, which  have the
  same  holomorph   as  $G$,  that   is,  such  that   $N_{S(G)}(N)  =
  N_{S(G)}(\rho(G))$,   where   $\rho$    is   the   (right)   regular
  representation.

  We give an alternative approach to  Mills' result, which relies on a
  characterization of the regular  subgroups of $N_{S(G)}(\rho(G))$ in
  terms of commutative ring structures on $G$.

  We  are led to  solve, for  the case  of a finitely  generated abelian
  group $G$, the following problem:  given an abelian group $(G, +)$,
  what are the  commutative ring structures $(G, +,  \cdot)$ such that
  all automorphism of $G$ as a  group are also automorphisms of $G$ as
  a ring?
\end{abstract}

\thanks{The first author acknowledges support by Dipartimento di Matematica,
  Università degli Studi di Trento.  The  authors are members
  of INdAM---GNSAGA, Italy}

\maketitle

\thispagestyle{empty}

\bibliographystyle{elsarticle-num}

%\tableofcontents

%\input{Introduction}
\section{Introduction}

Let  $G$ be  a  group, and  $\rho  :  G \to  S(G)$  its right  regular
representation, where $S(G)$  is the group of permutations  on the set
$G$. The normalizer 
\begin{equation*}
  \Hol(G) = N_{S(G)}(\rho(G))
\end{equation*}
of  the image  of $\rho$  is the  \emph{holomorph} of  $G$, and  it is
isomorphic to the  natural extension of $G$ by  its automorphism group
$\Aut(G)$.  It is well-known that $\Hol(G) = N_{S(G)}(\lambda(G))$,
where $\lambda : G \to S(G)$ is the left regular representation, since
$[\rho(G), \lambda(G)] = 1$.

The \emph{multiple  holomorph} of $G$ has  been defined in
G.A.~Miller   \cite{Miller-multi}  as  
\begin{equation*}
  N_{S(G)}(\Hol(G))  =
  N_{S(G)}(N_{S(G)}(\rho(G))).
\end{equation*}
Miller   has  shown  that   the  quotient group
\begin{equation*}
  T(G)
  =
  N_{S(G)}(\Hol(G)) / \Hol(G)
\end{equation*}
acts regularly  by
conjugation on  the set of the  regular subgroups $N$ of  $S(G)$ which
are isomorphic to  $G$ and have the same  holomorph as $G$, that
is, the  regular subgroups $N \cong G$ of  $S(G)$ such that
\begin{equation*}
  N_{S(G)}(N) = N_{S(G)}(\rho(G)).
\end{equation*}

There has been some attention in the recent literature
\cite{Kohl-multi} to the problem of determining, for $G$ in a given
class of groups, the set
\begin{equation*}
  \Hc(G)
  =
  \Set{ N \le S(G) : \text{$N$ is regular, $N \cong G$ and $N_{S(G)}(N)
    = \Hol(G)$} }
\end{equation*}
and the group $T(G)$.

In 1951 W.H.~Mills~\cite{Mills-multi} determined these data for
a finitely-generated abelian group $G$, extending the results
of~\cite{Miller-multi} for finite abelian  groups. Miller enumerated
the regular groups $N$ such that $N \cong G$ and $N_{S(G)}(N)
= \Hol(G)$. Mills noted that for $N$ abelian and regular, the
condition $N_{S(G)}(N) 
= \Hol(G)$ implies $N \cong G$. Later Mills showed in~\cite{Mills-non}
that if $G$ abelian, $N$ is a regular subgroup of $S(G)$, and $N_{S(G)}(N)
= \Hol(G)$, then $N$ is abelian as well.

In this paper, we redo
Mills' work using the approach of~\cite{affine}, which allows us to
translate the problem in terms of commutative rings. In particular, we
are led to solve the  following question, which
might be of independent interest, for the case when $G$ is a
finitely-generated abelian group.
\begin{question}
  Let $(G, +)$ be an abelian group. 

  What are the  commutative ring structures $(G, +,  \cdot)$ such that
  the automorphisms of $G$ as a group are also automorphisms of $G$ as
  a ring?
\end{question}

Theorem~\ref{thm:TG} states that if $G$
is a  finitely-generated abelian group,  then $T(G)$ is  an elementary
abelian $2$-group, of  order $1$, $2$, or $4$. In  other words, for a
given $G$,  there are either $1$,  $2$, or $4$  regular
subgroups $N$  of $S(G)$  that are  isomorphic to  $G$, and  such that
$N_{S(G)}(N) = N_{S(G)}(\rho(G))$.

These regular  subgroups are  described, via  the just  mentioned ring
connection, in Theorem~\ref{thm:main}.
In~\ref{thm:submain} we  also determine explicitly all
the group structures $G$ for which $\Size{T(G)} > 1$.

The     plan    of     the    paper     is    the     following.    In
Section~\ref{sec:same-holomorph} we define the various holomorphs, and
set  up  the problem.  In  Section~\ref{sec:regular}  we rephrase  the
problem in terms  of rings. The classification of the  rings is worked
out  in  Section~\ref{sec:class}. The  group  $T(G)$  is discussed  in
Section~\ref{sec:group}.

\section{Groups with the same holomorph}
\label{sec:same-holomorph}

In this section, $G$ is an additively written group.

The \emph{holomorph} of a group $G$ is the natural semidirect product 
\begin{equation*}
  \Aut(G) \, G
\end{equation*}
of $G$ by its automorphism group $\Aut(G)$. Let $S(G)$ be the group of
permutations on the set $G$. Consider the (right) regular representation
\begin{align*}
  \rho :\ &G \to S(G)
        \\&g \mapsto (x \mapsto x + g).
\end{align*}
The following is well-known.
\begin{proposition}\label{prop:right-and-left}
  $N_{S(G)}(\rho(G)) =  \Aut(G) \, \rho(G)$ is isomorphic  to the holomorph
  $\Aut(G) \, G$ of $G$.
\end{proposition}

\begin{definition}
  We write  $\Hol(G) = N_{S(G)}(\rho(G))$. We will  refer to either of
  the isomorphic groups
  $N_{S(G)}(\rho(G))$ and $\Aut(G) \, G$ as the \emph{holomorph} of $G$.
\end{definition}

One  may
inquire, what  are the regular subgroups  $N \le S(G)$ which  have the
same holomorph as $G$, that is, for which
\begin{equation}\label{eq:samehol}
  \Hol(N) \cong N_{S(G)}(N) = N_{S(G)}(\rho(G)) = \Hol(G).
\end{equation}

W.H.~Mills has noted in~\cite{Mills-multi} that if~\eqref{eq:samehol}
holds, then $G$ and $N$ need 
not be isomorphic. 

When we restrict our attention to  the regular subgroups $N$ of $S(G)$
for which $N_{S(G)}(N) = \Hol(G)$ and $N  \cong G$, we can appeal to a
result   of    G.A.~Miller~\cite{Miller-multi}.    Miller    found   a
characterization  of these  subgroups in  terms of  the \emph{multiple
  holomorph} of $G$
\begin{equation*}
  N_{S(G)}(\Hol(G)) = N_{S(G)}(N_{S(G)}(\rho(G))).
\end{equation*}
Consider the set
\begin{equation*}
  \Hc(G)
  =
  \Set{ N \le S(G) : \text{$N$ is regular, $N \cong G$ and $N_{S(G)}(N)
    = \Hol(G)$} }.
\end{equation*}
Using the well-known fact that two regular subgroups of $S(G)$ are
isomorphic if and only if they are conjugate in $S(G)$, Miller showed
that the group $N_{S(G)}(\Hol(G))$ acts transitively on $\Hc(G)$ by
conjugation. 
Clearly the 
stabilizer in $N_{S(G)}(\Hol(G))$ of any element $N \in \Hc(G)$ is $N_{S(G)}(N)
= \Hol(G)$. We obtain
\begin{theorem}
  The group
  \begin{equation*}
    T(G) = N_{S(G)}(\Hol(G)) / \Hol(G)
  \end{equation*}
  acts regularly on $\Hc(G)$ by conjugation.
\end{theorem}

\section{Regular normal subgroups of the holomorph}
\label{sec:regular}

Given an abelian group $G$, we aim first at giving a description of the set
\begin{equation*}
  \Kc(G)
  =
  \Set{ N \le S(G) : \text{$N$ is regular, $N \norm \Hol(G)$} }
  \supseteq \Hc(G).
\end{equation*}

It was noted in~\cite{FCC} that the results of~\cite{affine} on affine
groups admit a straightforward extension  to the case of holomorphs of
abelian groups.  We recall this here in our context.

Let $N \le \Hol(G)$ be a regular
subgroup. Write  $\nu(g)$, with $g \in  G$, for the unique  element of
$N$ such that $0^{\nu(g)} = g$. (We write group actions as exponents.)
Then there is a  map $\gamma : G \to \Aut(G)$ such that  for $g \in G$
we can write uniquely
\begin{equation}\label{eq:unique-form}
  \nu(g) = \gamma(g) \rho(g).
\end{equation}
For $g, h \in G$ we have
\begin{equation}\label{eq:nu}
  \nu(g) \nu(h)
  =
  \gamma(g) \rho(g) \gamma(h) \rho(h)
  =
  \gamma(g) \gamma(h) \rho(g^{\gamma(h)} + h).
\end{equation}
Since $N$ is a subgroup of $S(G)$, and the
expression~\eqref{eq:unique-form} is unique, we obtain, for $g, h \in G$,
\begin{equation}\label{eq:gamma}
  \gamma(g) \gamma(h) = \gamma( g^{\gamma(h)} + h ).
\end{equation}
Note, for later usage, that~\eqref{eq:gamma} can be
rephrased, setting $k = g^{\gamma(h)}$, as
\begin{equation}\label{eq:gamma-for-dot}
  \gamma(k + h) = \gamma(k^{\gamma(h)^{-1}}) \gamma(h),
\end{equation}
fo $h, k \in G$.

To enforce $N \norm  \Hol(G)$, it is now enough to  make sure that $N$
is normalized by  $\Aut(G)$. In fact, $N$ is a  transitive subgroup of
$\Hol(G)$ acting on $G$. Since $\rho(G)$ is regular, the stabilizer of
$0$ in  $\Hol(G) = \Aut(G)  \rho(G)$ is $\Aut(G)$. Hence  $\Hol(G)$ is
the product of $\Aut(G)$ and $N$ (and then it is a semidirect product,
as $N$ is regular). 

In order for $\Aut(G)$ to normalize $N$,
we  must  have  that for  all $\beta
\in  \Aut(G)$  and  $g  \in  G$,     the  conjugate
$\nu(g)^{\beta}$ of $\nu(g)$ by $\beta$ in $S(G)$ lies in $N$.  
Since
\begin{equation*}
  \nu(g)^{\beta}
  =
  (\gamma(g) \rho(g))^{\beta}
  =
  \gamma(g)^{\beta} \rho(g)^{\beta}
  =
  \gamma(g)^{\beta} \rho(g^{\beta}),
\end{equation*}
uniqueness of~\eqref{eq:unique-form} implies that this is equivalent to
\begin{equation}\label{eq:gamma-is-autinv}
   \gamma(g^{\beta}) = \gamma(g)^{\beta}
\end{equation}
for $g \in G$ and $\beta \in \Aut(G)$. Applying this
to~\eqref{eq:gamma-for-dot}, we obtain
\begin{equation}\label{eq:gamma-is-anti}
  \gamma(k + h)
  =
  \gamma(k^{\gamma(h)^{-1}}) \gamma(h)
  =
  \gamma(k)^{\gamma(h)^{-1}} \gamma(h)
  =
  \gamma(h) \gamma(k),
\end{equation}
that is, $\gamma : G \to \Aut(G)$ is a homomorphism, as $G$ is abelian.

Note that~\eqref{eq:gamma}  follows
from~\eqref{eq:gamma-is-autinv}~and \eqref{eq:gamma-is-anti}, 
as
\begin{equation*}
  \gamma( g^{\gamma(h)} + h )
  =
  \gamma(g)^{\gamma(h)} \gamma(h) 
  =
  \gamma(g) \gamma(h).
\end{equation*}

We now state the characterization we will be exploiting in the rest of
the paper.
\begin{theorem}\label{thm:normal-regular}
  Let $G$ be an abelian group. The following data are equivalent.
  \begin{enumerate}
  \item\label{item:ars}
    An abelian regular subgroup $N \norm \Hol(G)$, that is, an
    element of $\Kc(G)$. 
  \item\label{item:gamma} 
    A homomorphism 
    \begin{equation*}
      \gamma : G \to \Aut(G)
    \end{equation*}
    such that for $g \in G$ and
    $\beta \in \Aut(G)$
    \begin{equation}\label{eq:gamma-for-normal}
      \gamma(g^{\beta}) = \gamma(g)^{\beta}.
    \end{equation}
  \item\label{item:enters-ring} 
    A commutative  rings structure $(G, +,
    \cdot)$ such that
    \begin{enumerate}
       \item\label{item:isagroup}
         the operation $g \circ h = g + h + g h$ defines a group
      structure $(G, \circ)$,
    \item\label{item:ghk}
      $g h k = 0$ for all $g, h, k \in G$, and
    \item\label{item:auto-auto}
      each automorphism of the group $(G, +)$ is also an
      automorphism of the ring $(G, +, \cdot)$.
    \end{enumerate}
  \end{enumerate}
  Moreover, under these assumptions
  \begin{enumerate}[(i)]
  \item in terms of~\eqref{item:gamma}, the operations
    of~\eqref{item:enters-ring} are given by
    \begin{equation*}
      g \cdot h = -g + g^{\gamma(h)},
      \qquad\text{and}\qquad
      g \circ h = g^{\gamma(h)} + h.
    \end{equation*}
    for $g, h \in G$.
  \item\label{item:GcircisoN}
    The function $\nu$ of~\eqref{eq:nu} defines an isomorphism $(G,
    \circ) \to N$. 
  \item\label{item:auto}
    Every automorphism  of $G$  is also  an automorphism  of $(G,
    \circ)$.
  \end{enumerate}
\end{theorem}

Note that~\eqref{item:ghk} implies~\eqref{item:isagroup}. Note also
that~\eqref{item:GcircisoN} tells us that each ring structure as
in~\eqref{item:enters-ring} yields a distinct $N \in \Kc(G)$.

\begin{proof}
  We have already seen that~\eqref{item:ars}~and \eqref{item:gamma}  are
  equivalent.
  
  We  now recall  from~\cite{affine, FCC}  that  if $N$  is a  regular
  abelian  subgroup  of  $\Hol(G)$,  and $\gamma$  is  the  associated
  function as in~\eqref{item:gamma}, then, setting, for $g, h \in G$
  \begin{equation*}
    g \cdot h = -g + g^{\gamma(h)},
  \end{equation*}
  we obtain a ring structure $(G, +. \cdot)$ on $G$ such that
  \begin{equation*}
    g \circ h = g + h +  g h = g^{\gamma(h)} + h
  \end{equation*}
  defines a group structure $(G, \circ)$.
  
  To show that~\eqref{item:gamma} implies~\eqref{item:ghk}, we have to
  prove that for all $h, k \in G$ we have $\gamma(h k) = 1$. (This was
  already observed  in a  comment after Lemma~3  of~\cite{affine}.) In
  fact
  \begin{equation*}
    \gamma(h k) 
    = 
    \gamma(h^{\gamma(k)} - h)
    =
    \gamma(h)^{\gamma(k)} \gamma(h)^{-1}
    =
    [\gamma(k), \gamma(-h)]
    =
    1,
  \end{equation*}
  as $\gamma : G \to \Aut(G)$ is a homomorphism, and $G$ is abelian.

  To show that~\eqref{item:gamma} implies~\eqref{item:auto-auto}, let
  $h, k \in G$, and $\beta \in \Aut(G)$. We have
  \begin{multline*}
    h^{\beta} \cdot k^{\beta}
    =
    -h^{\beta} + h^{\beta \gamma(k^{\beta})}
    =
    -h^{\beta} + h^{\beta \gamma(k)^{\beta}}
    =\\=
    - h^{\beta} + h^{\gamma(k) \beta}
    =
    (-h + h^{\gamma(k)})^{\beta}
    =
    (h \cdot k)^{\beta},
  \end{multline*}
  where we have used~\eqref{eq:gamma-for-normal}.

  The bijection $\nu$ introduced above is a homomorphism
  $(G, \circ) \to N$ by~\eqref{eq:nu}~and \eqref{eq:gamma}.

  Finally, \eqref{item:auto} follows from
  \begin{equation*}
    (g \circ h)^{\beta}
    =
    (g + h + g h)^{\beta}
    =
    g^{\beta} + h^{\beta} + g^{\beta} h^{\beta}
    =
    g^{\beta} \circ h^{\beta}
  \end{equation*}
  for $g, h \in G$ and $\beta \in \Aut(G)$.

  Conversely, given a ring as in~\eqref{item:enters-ring}, the
  following calculations show that the function $\gamma : G \to S(G)$
  given by $\gamma(g) : h \mapsto h +  h g$
  satisfies the conditions of~\eqref{item:gamma}. (Here $\gamma(g) \in
  S(G)$ because $\gamma(g) \gamma(-g) : h \mapsto (h + h g) + (h + h
  g) (-g) = h + h(g -g) = h$, where we have used~\eqref{item:ghk}.)

  \begin{align*}
    (h + k)^{\gamma(g)}
    =
    h + k + (h + k) g
    =
    h + h g + k + k g
    =
    h^{\gamma(g)} + k^{\gamma(g)},
  \end{align*}
  for all $g, h, k \in G$,
  shows that $\gamma$ maps $G$ into $\Aut(G)$.

  \begin{align*}
    g^{\gamma(h) \gamma(k)}
    =
    (g + g h) + (g + g h) k
    =
    g + g(h + k)
    =
    g^{\gamma(h+k)},
  \end{align*}
  for all $g, h, k \in G$, where we have used~\eqref{item:ghk}, shows
  that $\gamma: G \to \Aut(G)$ is a homomorphism.

  \begin{align*}
    h^{\gamma(g^{\beta})}
    =
    h + h g^{\beta}
    =
    (h^{\beta^{-1}} + h^{\beta^{-1}} g)^{\beta}
    =
    h^{\beta^{-1} \gamma(g) \beta}
    =
    h^{\gamma(g)^{\beta}},
  \end{align*}
  for all $g, h \in G$, and $\beta \in \Aut(G)$, shows that $\gamma$
  satisfies~\eqref{eq:gamma-for-normal}.
\end{proof}

Suppose the finitely generated abelian group $(G, +)$ admits a ring
structure $(G, +, \cdot)$ as in
Theorem~\ref{thm:normal-regular}.\eqref{item:enters-ring}. 
Taking $\beta \in \Aut(G, +)$ to be inversion $g \mapsto  - g$, we get
that for all $g, h \in G$ one has
\begin{equation*}
  - g h = (- g) (- h) = g h,
\end{equation*}
that is, all products satisfy
\begin{equation}\label{eq:twice-a-product}
  2 \cdot g h = 0.
\end{equation}
%% It follows that is
%% there are no elements of order $2$ in $G$, then ring multiplication is
%% trivial, and $N = G$.

We have obtained
\begin{lemma}\label{lemma:threefold}
  In the commutative ring $(G, + , \cdot)$ as
  in~Theorem~\ref{thm:normal-regular}.\eqref{item:enters-ring} we have  
  \begin{enumerate}
  \item\label{item:products-are-involutions} 
    $2 \cdot g h = 0$,   for all $g, h \in G$, so that
  \item if $(G, +)$ has no elements of order $2$, ring multiplication
    is trivial, and
  \item\label{item:Frobenius}
    $(g + h)^{2} = g^{2} + h^{2}$ ,   for all $g, h \in G$.
  \end{enumerate}
\end{lemma}

\section{The classification}
\label{sec:class}

From now on, let $G$ be a finitely generated abelian
group. 

For such $G$, Mills~\cite{Mills-multi} has determined the set
\begin{equation*}
  \Hc(G)
  =
  \Set{ N \le S(G) : \text{$N$ is regular, $N \cong G$ and $N_{S(G)}(N)
      = \Hol(G)$} }.
\end{equation*}
In the following we will first determine the set
\begin{equation*}
  \Kc(G)
  =
  \Set{ N \le S(G) : \text{$N$ is regular, $N \norm \Hol(G)$} }
  \supseteq
  \Hc(G),
\end{equation*}
weeding  out in the process the $N$ for which $N_{S(G)}(N) >
N_{S(G)}(\rho(G))$. In 
Section~\ref{sec:group} we will show that the remaining groups are
precisely the elements of $\Hc(G)$, and 
we will also determine the group $T(G)$.

According to Theorem~\ref{thm:normal-regular}, we  proceed to find
all ring structures $(G, +, \cdot)$ such that all automorphisms of $G$
as a group  are also automorphisms of  $G$ as a ring.  We will usually
tacitly ignore the trivial case when $G^{2} = \Set{ x y : x, y \in G }
= \Set{ 0 }$.

Write
\begin{equation*}
  G = F \times H \times K,
\end{equation*}
where $F$ is  free abelian of finite rank, $H$ is a finite $2$-group,
and $K$ is a finite group of odd order.

If $a \in K$ has odd order $d$, then according
to Lemma~\ref{lemma:threefold}\eqref{item:products-are-involutions},
for all $b \in G$ we have $a b = d (a 
b) = (d a) b = 0$.  Therefore the odd part $K$ lies in the
annihilator. For the ring structure on $G$ to be
non-trivial, it has thus to be non-trivial on $F \times H$.
Because of this, from now on we will assume
\begin{equation*}
  G = F \times H,
\end{equation*}
where $F$ is  free abelian of finite rank, and  $H$, the torsion part,
is a finite $2$-group. We write
\begin{equation*}
  \Omega(H)
  =
  \Set{ t \in H : 2 t = 0 }.
\end{equation*}
By  Lemma~\ref{lemma:threefold}.\eqref{item:products-are-involutions},
all  products in  the ring  $(G, +,  \cdot)$ lie  in $\Omega(H)$.  
We
regard $\Omega(H)$ as a vector space  over the field $\E = \Set{0, 1}$
with $2$ elements.

\subsection{The case $F = 0$.}\ppar
\label{sub:torsion}

We first discuss the structure of the $2$-torsion part $H$ in the case
when the torsion-free part $F$ is zero.

Write
\begin{equation*}
  H = \prod_{i = 1}^{m} \Span{x_{i}},
\end{equation*}
where $\Size{x_{i}} = 2^{e_{i}}$, with  $e_{i} > 0$, and $\Size{x_{i}}
\ge \Size{x_{j}}$ for  $i \le j$. Write $t_{i} =  2^{e_{i} - 1} x_{i}$
for the involution in $\Span{x_{i}}$.

A \emph{homogeneous component} of $H$ will be a subgroup $\prod_{i =
  a}^{b} \Span{x_{i}}$, for some $a \le b$, such that $\Size{x_{a-1}}
> \Size{x_{a}} = \Size{x_{a+1}} = \dots = \Size{x_{b}} >
\Size{x_{b+1}}$, where the first inequality does not occur if $a = 1$,
and the last one does not occur if $b = m$.

Consider the following automorphisms of $H$.
\begin{enumerate}
\item $\xi_{i j}$, for $x_{i}, x_{j}$ in the same homogeneous component,
  exchanges $x_{i}$ with $x_{j}$, and
  leaves all the other $x_{k}$ fixed.
\item $\gamma_{i j}$, for $i < j$, maps $x_{i}$ to $x_{i} + x_{j}$,
  and leaves all the other $x_{k}$ fixed.
\item $\beta_{i j}$, for $i > j$,  maps $x_{i}$ to $x_{i} +
  2^{e_{j} - e_{i}} x_{j}$,
  and leaves all the other $x_{k}$ fixed. Note that $\beta_{i j}$
  maps $t_{i}$ to $t_{i} + t_{j}$.
\end{enumerate}

\begin{proposition}
  \label{subsub:greaterthan4}
  If $m = 1$, then we have the following possibilities.
  \begin{enumerate}
  \item $x_{1}^{2} = 0$, and then multiplication is trivial. This is
    always the case when $\Size{x_{1}} = 2$.
  \item $x_{1}^{2} = t_{1}$, and then
    \begin{enumerate}
    \item if $\Size{x_{1}} = 4$, then $N \notin \Hc(G)$;
    \item if $\Size{x_{1}} > 4$, then $N \in \Hc(G)$.
    \end{enumerate}
  \end{enumerate}
\end{proposition}

\begin{proof}
  If ring multiplication is non-trivial, then $x_{1}^{2} = t_{1} \ne 0$.
  
  If $\Size{x_{1}} = 2$, we obtain $t_{1}^{2} = t_{1}$, and thus
  $t_{1}^{3} = t_{1}$, contradicting
  Theorem~\ref{thm:normal-regular}\eqref{item:ghk}.
  
  If $\Size{x_{1}} = 4$, we obtain $x_{1} \circ x_{1} = 2 x_{1} + t_{1}
  = 0$, so $(H, \circ)$ is elementary abelian of order $4$, so that
  its automorphism group is larger than that of $H$, and $N_{S(G)}(N)
  > N_{S(G)}(\rho(G))$. Therefore $N \notin \Hc(G)$.

  If $\Size{x_{1}} > 4$, then one sees immediately that $x_{1}$
  retains its order in $(H, \circ)$.
\end{proof}

In this case we have thus two rings.
\begin{equation}\label{eq:this-is-cyclic}
  \begin{cases}
    n = 0, m = 1\\
    \Size{x_{1}} > 4\\
    x_{1}^{2} \in \Set{0, t_{1}}\\
  \end{cases}
\end{equation}

We now turn  to the case $m >  2$. We will see that in  most cases the
number of ring structures depends only on the orders of $x_{1}, x_{2}$ (and
possibly   $x_{3}$),  and   their  relationships.   An  exception   is
case~\ref{eq:all-zero-but-FH},  where  the  orders of  $x_{1},  \dots,
x_{k}$ matter, for an arbitrary $k \le m$.

\begin{rec}\label{rec:x1x2}
  If $m \ge 2$, then
  \begin{equation*}\label{eq:a-product}
    x_{1} x_{2} 
    = 
    \eta_{1} t_{1} + \eta_{2} t_{2} \ne t_{2},
    \quad\text{for some $\eta_{i} \in \E$.}
  \end{equation*}
\end{rec}

\begin{proof}
  We have $x_{1} x_{2} = \sum_{k=1}^{m} \eta_{k} t_{k}$ for some
  $\eta_{k} \in \E$. 
  Applying $\beta_{i1}$ to this, for $i > 2$, we see that
  \begin{equation*}
    x_{1} x_{2} = (x_{1} x_{2}) \beta_{i 1}
    =
    (\sum_{k=1}^{m} \eta_{k} t_{k}) \beta_{i 1}
    =
    (\sum_{k=1}^{m} \eta_{k} t_{k}) + \eta_{i} t_{1},
  \end{equation*}
  whence $\eta_{i} = 0$ for $i > 2$.

  It remains to show that $x_{1} x_{2} \ne t_{2}$. If $x_{1} x_{2} =
  t_{2}$, in the case when $\Size{x_{1}} = \Size{x_{2}}$ we have
  $x_{1} x_{2} = (x_{1} x_{2}) \xi_{12} = t_{1}$, a contradiction;
  when $\Size{x_{1}} > \Size{x_{2}}$, that is, $e_{1} > e_{2}$ we have
  $t_{1} + t_{2} = (t_{2}) \beta_{21} = (x_{1}
  x_{2}) \beta_{21} = x_{1} (x_{2} + 2^{e_{1} - e_{2}} x_{1}) = x_{1}
  x_{2}$, a contradiction, using the fact that
  Lemma~\ref{lemma:threefold}  implies $ 2^{e_{1} - e_{2}}
  x_{1}^{2} = 0$.
\end{proof}

\begin{rec}\label{rec:all-zero}
  Suppose  $m > 2$,  $\Size{x_{1}} > \Size{x_{3}}$, and
  Suppose  $m > 2$,  $\Size{x_{1}} > \Size{x_{3}}$, and
  either $\Size{x_{2}} >  \Size{x_{3}}$,
  or $\eta_{2} = 0$ in Fact~\ref{rec:x1x2}.
  
  Then $x_{1} x_{j} = x_{2} x_{k} = x_{k} x_{j} = 0$ for $k, j
  > 2$.
\end{rec}

\begin{proof}
  Let $k, j > 2$. 

  Note  that $\Size{x_{1}}  > \Size{x_{3}}  \ge  \Size{x_{k}}$ implies
  \begin{equation*}
    (t_{1}) \gamma_{1k}
    =
    (2^{e_{1} - 1} x_{1}) \gamma_{1k}
    =
    2^{e_{1} - 1} (x_{1} + x_{k})
    =
    2^{e_{1} - 1} x_{1}
    = 
    t_{1}.
  \end{equation*}
  Similarly, if  $\Size{x_{2}} > \Size{x_{3}}$ we  have $(t_{2}) \gamma_{2
    j} = t_{2}$. Thus under the given hypotheses we have, for $\eta_{1},
  \eta_{2} \in \E$,
  \begin{equation*}
    (\eta_{1}  t_{1} +  \eta_{2}  t_{2}) \gamma_{1i}
    =
    \eta_{1}  t_{1} +  \eta_{2}  t_{2}
    =
    (\eta_{1}  t_{1} +  \eta_{2}  t_{2}) \gamma_{2 j},
  \end{equation*}
  where the last equality depends on the fact that by assumption
  either $\Size{x_{2}} > \Size{x_{3}}$, and thus $(t_{2}) \gamma_{2
    j} = t_{2}$, or $\eta_{2} = 0$.
  
  Apply $\gamma_{1k}$ to $x_{1} x_{2}$, and use \Fact~\ref{rec:x1x2}, to get
  \begin{equation*}
    x_{1}  x_{2}  = 
    (x_{1} x_{2}) \gamma_{1k} = (x_{1}  + x_{k}) x_{2} = x_{1} x_{2} +
    x_{k} x_{2},
  \end{equation*}
  whence $x_{k} x_{2} = 0$.
  
  Apply $\gamma_{2 j}$ to $x_{1} x_{2}$ to get
  \begin{equation*}
    x_{1}  x_{2}  = 
    (x_{1} x_{2}) \gamma_{2 j} = x_{1} (x_{2}   + x_{j}) = x_{1} x_{2} +
    x_{1} x_{j},
  \end{equation*}
  whence $ x_{1} x_{j} = 0$.
  
  Finally, apply $\gamma_{1k} \gamma_{2j}$ to $x_{1} x_{2}$ to get
  \begin{equation*}
    x_{1}  x_{2}  = 
    (x_{1} x_{2}) \gamma_{1k}  \gamma_{2j}= (x_{1}  + x_{k}) (x_{2}
    + x_{j})= x_{1} x_{2} + 
    x_{k} x_{j},
  \end{equation*}
  whence
  $x_{k} x_{j} = 0$.
\end{proof}

\subsubsection{Torsion case, $m \ge 2$, $x_{1} x_{2} = t_{1}$}

\ppar

If $\Size{x_{1}} = \Size{x_{2}}$, applying $\xi_{12}$ to $x_{1} x_{2}
= t_{1}$ we get $x_{1} x_{2} = t_{2}$, a contradiction.

If $\Size{x_{1}} >  \Size{x_{2}}$, we have $2^{e_{1} - 1}  x_{2} = 0$,
so that $t_{1} = 2^{e_{1} - 1} x_{1} = 2^{e_{1} - 1} (x_{1} + x_{2}) =
2^{e_{1} - 1} (x_{1}) \gamma_{12}  = (t_{1}) \gamma_{12}$, which implies
$x_{1} x_{2} = t_{1} = t_{1} \gamma_{12} = (x_{1} x_{2}) \gamma_{12} =
(x_{1} +  x_{2}) x_{2} =  x_{1} x_{2}  + x_{2}^{2}$, so that  $x_{2}^{2} =
0$. Applying $\gamma_{2i}$ to the  last identity, we obtain $x_{i}^{2}
= 0$ for $i > 2$.

Using \Fact~\ref{rec:all-zero} we  obtain $x_{1} x_{j} =  x_{2} x_{i} =
x_{i} x_{j} = 0$ for $i, j > 2$.

If $m > 2$ and $\Size{x_{2}} = \Size{x_{3}}$, we have 
$t_{1} = (t_{1}) \xi_{23} = (x_{1} x_{2}) \xi_{23} = x_{1} x_{3}$, a
contradiction.

We have obtained the following result.
\begin{proposition}
  The following rings share the same
  group structure, and give rise to groups $(G, \circ) \cong G$.
  \begin{equation}
    \label{eq:5-two-or-four}
    \begin{cases}
      n = 0, m \ge 2\\
      \Size{x_{1}} > \Size{x_{2}}\\
      \Size{x_{1}} > 4 & \text{if $x_{1}^{2} \ne 0$}\\
      \Size{x_{2}} > \Size{x_{3}} & \text{if $m > 2$ and $x_{1} x_{2} \ne 0$}\\
      x_{1}^{2} \in \Set{ 0, t_{1} }\\
      x_{i}^{2} = 0, &\text{for $i > 1$}\\
      x_{1} x_{2} \in \Set{0, t_{1}}\\
      x_{1} x_{i} = x_{2} x_{j} = x_{i} x_{j} = 0, &\text{for $i, j >
        2$}\\
    \end{cases}
  \end{equation}
  Here there are either four rings, two rings, or just the trivial ring:
  \begin{itemize}
  \item if $\Size{x_{1}} > 4$, and either $m = 2$, or $m > 2$ and
    $\Size{x_{2}} > \Size{x_{3}}$, there are four rings;
  \item if $\Size{x_{1}} = 4$ and $m > 2$ there is just the trivial ring;
  \item there are two rings in the remaining cases.
  \end{itemize}
\end{proposition}
Note that we have required $\Size{x_{1}} > 4$ if $x_{1}^{2} = t_{1} \ne 0$, as
in Proposition~\ref{subsub:greaterthan4}, to make sure that $x_{1}$
retains its 
order in $(H, \circ)$.

\begin{remark}\label{rem:is-iso}
  Clearly   the  elements   $x_{i}$  retain   their  orders   in  $(H,
  \circ)$. Moreover it  is easy to see that $(H,  \circ)$ is still the
  direct product of the subgroups spanned by the $x_{i}$, so that $(H,
  \circ)$ is isomorphic to $H$.
\end{remark}

\subsubsection{Torsion case, $m \ge 2$, $x_{1} x_{2} = t_{1} + t_{2}$}

\ppar

If $\Size{x_{1}} > \Size{x_{2}}$, then 
$t_{1} + t_{2} = x_{1} x_{2} = (x_{1} x_{2}) \beta_{21} = (t_{1} +
t_{2}) \beta_{21} = t_{2}$, a contradiction.

Therefore $\Size{x_{1}} = \Size{x_{2}}$. Applying $\gamma_{12}$ and
$\beta_{12}$ to $x_{1} x_{2} = t_{1} + t_{2}$ we obtain $x_{1}^{2} =
t_{1}$ and $x_{2}^{2} = t_{2}$. 

When $m > 2$, if $\Size{x_{2}} = \Size{x_{3}}$, applying $\xi_{23}$ to
$x_{1} x_{2} = t_{1} + t_{2}$ we obtain $x_{1} x_{3} = t_{1} + t_{3}$.
But then  applying $\gamma_{23}$ to $x_{1}  x_{2} = t_{1} +  t_{2}$ we
obtain $t_{2} +  t_{3} = x_{1} x_{2}  + x_{1} x_{3} = t_{1}  + t_{2} +
t_{3}$, a contradiction.

Therefore  if $m  > 2$  we have  $\Size{x_{2}} >  \Size{x_{3}}$. Using
\Fact~\ref{rec:all-zero} we obtain  $x_{1} x_{j} = x_{2}  x_{i} = x_{i}
x_{j} = 0$ for $i, j > 2$. 
We have obtained the following result.
\begin{proposition}
  The following ring gives rise to a group $(G, \circ) \cong G$.
  \begin{equation}\label{eq:most-complicated}
    \begin{cases}
      n = 0, m \ge 2\\
      \Size{x_{1}} = \Size{x_{2}} > 4\\
      \Size{x_{2}} > \Size{x_{3}}\ \text{if $m > 2$}\\
      x_{1}^{2} = t_{1}, x_{2}^{2} = t_{2}\\
      x_{1} x_{2} = t_{1} + t_{2}\\
      x_{1} x_{i} = x_{2} x_{j} = x_{i} x_{j} = 0, &\text{for $i, j > 2$}\\
    \end{cases}
  \end{equation}
  The same group obviously allows also trivial ring multiplication.
\end{proposition}

Note that we have to take $\Size{x_{1}} > 4$ here, for the same
argument of Proposition~\ref{subsub:greaterthan4}. And $(H, \circ)$ is
isomorphic 
to $H$, as per
Remark~\ref{rem:is-iso}. 

\subsubsection{Torsion case, $m \ge 2$, $x_{1} x_{2} = 0$}

\ppar

Applying $\gamma_{12}$ to $x_{1} x_{2} = 0$ we obtain $x_{2}^{2} = 0$.

The  arguments of  the proof  of \Fact~\ref{rec:all-zero}  yield $x_{1}
x_{j} = x_{2} x_{i} = x_{i} x_{j} = 0$ for $i, j > 2$.

If $\Size{x_{1}} = \Size{x_{2}}$, then applying $\xi_{12}$ to
$x_{2}^{2} = 0$ we obtain $x_{1}^{2} = 0$, and the ring has trivial
multiplication. 

Therefore $\Size{x_{1}} > \Size{x_{2}}$, and the ring is described
in~\eqref{eq:5-two-or-four}. 

\subsection{The case $F \ne 0$}\ppar

Write
\begin{equation*}
  F = \prod_{i=1}^{n} \Span{z_{i}},
\end{equation*}
with all $z_{i} \ne 0$. 

Consider the following
automorphisms of $F \times H$, which are 
trivial on $H$.
\begin{enumerate}
\item $\Xi_{i j}$, for $i \ne j$ exchanges $z_{i}$ with $z_{j}$, and
  leaves $H$ and all the other $z_{k}$ fixed.
\item $\Gamma_{i j}$, for $i \ne j$, maps $z_{i}$ to $z_{i} + z_{j}$,
  and leaves $H$ and all the other $z_{k}$ fixed.
\item $\zeta_{i g}$, for $g \in H$, maps $z_{i}$ to $z_{i} + g$,
  and leaves $H$ and  all the other $z_{k}$ fixed.
\end{enumerate}

Recall that $\Omega(H) = \Set{ t \in G : 2 t = 0 }$. By
Lemma~\ref{lemma:threefold}, the ring product on $G$ yields a bilinear map
\begin{equation*}
  G / 2 G \times G / 2 G \to \Omega(H).
\end{equation*}
Since the  automorphisms $\Xi_{i j},  \Gamma_{i j}, \zeta_{i  g}, \xi_{i
  j}, \gamma_{i j}, \beta_{i j}$ generate all the automorphisms of $F
/ 2F \times H / 2 H$, it will be easy  to see that all rings $(G, +, \cdot)$
constructed  in  the   following  have  the  property   that  all  the
automorphisms  of the  group $(G,  +)$ are  also automorphisms  of the
ring.

Since by Lemma~\ref{lemma:threefold}
the square map $z \to z^{2}$ is a group
homomorphism $F \to \Omega(H)$, we may make the following
\begin{ass}\label{ass:ass}
  The  indexing of the $z_{i}$ is chosen so that if some square of
  the $z_{i}$ is non-zero, then $z_{1}^{2} \ne 0$.
\end{ass}

We first record the following well-know fact, which in our context can
be seen using the $\beta_{i1}$.
\begin{lemma}\label{lemma:min-char}
  Let $P \ne 1$ be a finite, abelian $p$-group. The following are
  equivalent:
  \begin{enumerate}
  \item \emph{$P$ has a characteristic minimal subgroup}, that is, a
    characteristic subgroup of order $p$,
  \item $P$ has a \emph{unique} characteristic minimal  subgroup, and
  \item $P$ is  the direct product of a cyclic group  of order $p^{e}$, for
    some $e \ge 1$, by a group of exponent less than $p^{e}$.
  \end{enumerate}
  If these conditions are verified, the unique characteristic
  subgroup of order $p$ is $P^{p^{e-1}}$.
\end{lemma}

Write
\begin{equation*}
  F^{2} = \Set{ a b : a, b \in F} \subseteq \Omega(H)
\end{equation*}
for the set of products of elements of $F$.

\begin{rec}\label{rec:char-inv}\label{rec:squares}\ppar
  \begin{enumerate}
  \item\label{item:char}  $\Aut(H)$ acts trivially on the set $F^{2}$.
  \item\label{item:ris1} The set $F^{2} \subseteq
    \Omega(H)$ is either zero or
    $\Span{t_{1}}$. If it is non-zero, then it is the unique minimal
    characteristic subgroup of $H$.
  \item\label{item:squares} The set $\Set{ z^{2} : z \in F } \subseteq
    \Omega(H)$ is either zero or
    $\Span{t_{1}}$. If it is non-zero, then it is the unique minimal
    characteristic subgroup of $H$, and $n = 1$.
  \end{enumerate}
\end{rec}

Note that when $F^{2} \ne 0$ in~\ref{item:ris1} and $\Set{ z^{2} : z
  \in F} \ne 0$ in~\ref{item:squares}, then by
Lemma~\ref{lemma:min-char} either $m = 1$ or $\Size{x_{1}} >
\Size{x_{2}}$.

\begin{proof}
  To see~\eqref{item:char}, take an arbitrary automorphism of $H$, and
  extend it trivially to $F$.

  Let $u$ be an arbitrary non-zero element of $F^{2}$. Then $\Span{u}$
  is a characteristic minimal subgroup of $H$, so that by
  Lemma~\ref{lemma:min-char} it is the unique characteristic minimal
  subgroup of $H$. This shows~\eqref{item:ris1}. 

  A similar argument
  yields the first part of~\eqref{item:squares}. If $\Set{ z^{2} : z
    \in F } \ne \Set{ 0 }$, then by Assumption~\ref{ass:ass} we have
  $z_{1}^{2} \ne 0$, and thus $z_{1}^{2} = t_{1}$. If $n > 1$,
  applying $\Xi_{12}$ we see that $z_{2}^{2} = t_{1}$, but then
  applying $\Gamma_{12}$ to $z_{1}^{2} = t_{1}$ we obtain
  \begin{equation*}
    t_{1} = z_{1}^{2} = (z_{1} + z_{2})^{2} = z_{1}^{2} + z_{2}^{2} =
    2 t_{1} = 0,
  \end{equation*}
  a contradiction.
\end{proof}

\begin{rec}\label{rec:Fne0}
  If $F \ne 0$, then $H^{2} = 0$.
\end{rec}

\begin{proof}
  Consider arbitrary $i, j \le m$. Since $z_{1} x_{i} \in
  H$, it is fixed by $\zeta_{1x_{j}}$, so that
  \begin{equation*}
    z_{1} x_{i} 
    = 
    (z_{1} x_{i}) \zeta_{1x_{j}} 
    = 
    (z_{1} + x_{j}) x_{i}
    =
    z_{1} x_{i} + x_{j} x_{i}, 
  \end{equation*}
  and thus $x_{i} x_{j} = 0$.
\end{proof}

\begin{rec}\label{rec:ngt2fsq0}
  If $n > 2$, then $F^{2} = 0$.
\end{rec}

\begin{proof}
  Let $i, j, k$ be distinct indices. Applying $\Gamma_{k j}$ to $z_{i}
  z_{k}$, we get
  \begin{equation*}
    z_{i} z_{k} = z_{i} (z_{k} + z_{j}) = z_{i} z_{k} + z_{i} z_{j},
  \end{equation*}
  whence $z_{i} z_{j} = 0$ for all $i, j$.
\end{proof}

If  $z_{1}  x_{1}  =  0$, then  applying  the
$\Xi_{1i}$ and the $\gamma_{1j}$ we see that $z_{i} x_{j} = 0$ for all
$i, j$, that is, $F H = 0$. 

Let us first consider the case when $F H \ne 0$, so that
\begin{equation*}
  z_{1} x_{1}
  =
  \sum_{k=1}^{n} \epsilon_{k} t_{k}
  \ne 0.
\end{equation*}
Applying $\beta_{i1}$ to  this, for $i > 1$, we  obtain $z_{1} x_{1} =
\sum_{k=1}^{n}  \epsilon_{k}  t_{k}  + \epsilon_{i}  t_{1}$,  so  that
$\epsilon_{i} =  0$, and thus  $z_{1} x_{1}  = t_{1}$. Note  that this
implies $0 = z_{1}^{2} x_{1} = z_{1} t_{1}$, so that $\Size{x_{1}} \ge
4$. If $n  > 1$, applying $\Xi_{12}$  to $z_{1} x_{1} =  t_{1}$ we get
$z_{2}  x_{1} =  t_{1}$, and  applying $\Gamma_{12}$  we get  $t_{1} =
(t_{1}) \Gamma_{12} = (z_{1} x_{1})  \Gamma_{12} = (z_{1} + z_{2}) x_{1}
= t_{1}  + t_{1}  = 0$, a  contradiction. Therefore $n  = 1$.  We have
obtained
\begin{rec}
  If $F H \ne 0$, then $n = 1$ and $z_{1} x_{1} = t_{1}$.
\end{rec}

Applying the
$\xi_{1j}$, we obtain that if
\begin{equation*}\label{eq:FH-not-zero-1}
  \Size{x_{1}} = \Size{x_{2}} = \dots = \Size{x_{k}} > \Size{x_{k+1}}
\end{equation*}
(where we might have $k = m$, so that the final inequality does not occur),
then
\begin{equation*}\label{eq:FH-not-zero-2}
  z_{1} x_{i} = t_{i}
  \text{, for $i \le k$},
  \qquad
  z_{1} x_{i} = 0 
  \text{, for $i > k$.} 
\end{equation*}
If $k > 1$, this implies $F^{2} = 0$, by \Fact~\ref{rec:char-inv} and
Lemma~\ref{lemma:min-char}.   
Also, if $n > 2$, then $F^{2} = F H = H^{2} = 0$.

We are now able to discuss the possibilities for the products on $F$. 

\subsubsection{$F \ne 0$, $F^{2} = 0$}\ppar

In this case we have
\begin{proposition}
  The following ring gives rise to a group $(G, \circ) \cong G$.
  \begin{equation}\label{eq:all-zero-but-FH}
    \begin{cases}
      n = 1, m \ge 1\\
      \Size{x_{1}} = \Size{x_{2}} = \dots = \Size{x_{k}} \ge 4
      &
      \text{for some $k \le m$}\\
      \Size{x_{k}} > \Size{x_{k+1}}, & \text{if $k < m$}\\
      z_{1}^{2} = 0\\
      z_{1} x_{i} = t_{i}, &
      \text{for $i \le k$}\\
      z_{1} x_{i} = 0, & 
      \text{for $i > k$}\\
      x_{i} x_{j} = 0, & \text{for all $i, j$}\\
    \end{cases}
  \end{equation}  
  The same group obviously allows also trivial ring multiplication.
\end{proposition}

\subsubsection{$F^{2} \ne 0$, $z_{1}^{2} \ne 0$}\ppar
\label{sec:to-forward}

By \Fact~\ref{rec:squares}\eqref{item:squares}, we have $n = 1$ here, and
\begin{equation*}
  z_{1}^{2} = t_{1},
\end{equation*}
with $\Span{t_{1}}$ characteristic in $H$. We have obtained the following.
\begin{proposition}
  The  following  rings   give  rise  to  groups   $(G,  \circ)  \cong
  G$.
  \begin{equation}\label{eq:tf-nis1}
    \begin{cases}
      n = 1, m \ge 1\\
      \text{$m = 1$, or $m > 1$ and $\Size{x_{1}} > \Size{x_{2}}$}\\
      \Size{x_{1}} \ge 4 & \text{if $z_{1} x_{1} \ne 0$}\\
      z_{1}^{2} \in \Set{0, t_{1}}\\
      z_{1} x_{1} \in \Set{0, t_{1}}\\
      x_{i} x_{j} = 0 & \text{for all $i, j$}\\
    \end{cases}
  \end{equation}
  These are two rings if $\Size{x_{1}} = 2$ (and then $z_{1} x_{1} = 0$,
  with $x_{1} = t_{1}$), four rings if $\Size{x_{1}} \ge 4$.
\end{proposition}
\begin{remark}\label{rem:isom}
Note that this case
comprises~\eqref{eq:all-zero-but-FH} when $k = 1$
in~\eqref{eq:all-zero-but-FH}. 
\end{remark}

\subsubsection{$F^{2} \ne 0$, $z_{1}^{2} = 0$}\ppar

By \Fact~\ref{rec:ngt2fsq0}, we have $n \le 2$.

The case $n = 1$ does not occur, as it means $F^{2} = 0$ here.

If $n = 2$, we have
\begin{equation*}%%%\label{eq:free-rank2}
  z_{1} z_{2} = t_{1}
\end{equation*}
by \Fact~\ref{rec:char-inv}\eqref{item:ris1}. We have obtained the following.
\begin{proposition}
  The  following  two rings   give  rise  to  groups   $(G,  \circ)  \cong
  G$.
  \begin{equation}\label{eq:r0n2}
    \begin{cases}
      n = 2, m \ge 1\\
      \Size{x_{1}} > \Size{x_{2}} & \text{if $m > 1$}\\ 
      z_{1}^{2} = z_{2}^{2} = 0\\
      z_{1} z_{2} \in \Set{0, t_{1}}\\
      z_{i} x_{j} = 0, & \text{for all $i, j$}\\
      x_{i} x_{j} = 0, & \text{for all $i, j$}\\
    \end{cases}
  \end{equation}
\end{proposition}

In all  of these cases, it  is easy to  see that $H$ is  isomorphic to
$(H, \circ)$, as per Remark~\ref{rem:is-iso}.

We can sum up the results of this section in the following theorems
which represents our main results.
\begin{theorem}\label{thm:main}
  Let $(G, +)$ be a finitely generated abelian group,
  \begin{equation*}
    G = F \times H,
  \end{equation*}
  where
  \begin{equation*}
    F = \prod_{i=1}^{n} \Span{z_{i}}
  \end{equation*}
  is torsion-free, of rank $n$,
  \begin{equation*}
    H = \prod_{i=1}^{m} \Span{x_{i}}
  \end{equation*}
  is a $2$-group, with $\Size{x_{1}} \ge \Size{x_{2}} \ge \dots \ge
  \Size{x_{m}} > 1$.

  The possible ring structures with non-trivial multiplication $(G, +,
  \cdot)$ on $(G, +)$, such that
  \begin{enumerate}
  \item $(G, +) \cong (G, \circ)$, and
  \item all automorphisms of $(G, +)$ are also automorphisms of $(G, +, \cdot)$
  \end{enumerate}
  are those listed under
  \begin{center}
    \eqref{eq:this-is-cyclic},
    \eqref{eq:5-two-or-four},
    \eqref{eq:most-complicated},
    \eqref{eq:all-zero-but-FH},
    \eqref{eq:tf-nis1},
    \eqref{eq:r0n2}.
  \end{center}
  The groups from the different cases
  are            pairwise            non-isomorphic,            except
  for~\eqref{eq:all-zero-but-FH}~and  \eqref{eq:tf-nis1}, as  noted in
  Remark~\ref{rem:isom}.

  In the cases~\eqref{eq:5-two-or-four}~and \eqref{eq:tf-nis1} we have
  two or four  rings (including the ring  with trivial multiplication)
  for the same group structure, in the other cases we have two.

  All of these $G$ can be enlarged to $G \times D$, where $D$ is an
  abelian group of odd order which lies in the annihilator of the ring.
\end{theorem}

\begin{theorem}\label{thm:submain}
  In the notation of Theorem~\ref{thm:main},  
  \begin{enumerate}
  \item 
  the groups $G$ such that
  $\Size{\Hc(G)} = \Size{T(G)} = 4$ are the following.
  \begin{align*}
  &\begin{cases}
    n = 0, m \ge 2\\
   \Size{x_{1}} > 4\\
   \Size{x_{1}} > \Size{x_{2}}\\
   \Size{x_{2}} > \Size{x_{3}} & \text{if $m > 2$}\\
  \end{cases}
  \\
  &\begin{cases}
    n = 1, m \ge 1\\
    \Size{x_{1}} \ge 4\\
    \text{$m = 1$, or $m > 1$ and $\Size{x_{1}} > \Size{x_{2}}$}\\
  \end{cases}
  \end{align*}
  \item 
  the groups $G$ such that
  $\Size{\Hc(G)} = \Size{T(G)} = 2$ are the following.
  \begin{align*}
  &\begin{cases}
    n = 0, m = 1\\
    \Size{x_{1}} > 4\\
  \end{cases}
  \\&
  \begin{cases}
    n = 0, m = 2\\
    \Size{x_{1}} = 4\\
    \Size{x_{2}} = 2\\
  \end{cases}
\\&
  \begin{cases}
    n = 0, m > 2\\
    \Size{x_{1}} > 4\\
    \Size{x_{1}} > \Size{x_{2}}\\
    \Size{x_{2}} = \Size{x_{3}}\\
  \end{cases}
\\&
  \begin{cases}
    n = 0, m \ge 2\\
    \Size{x_{1}} = \Size{x_{2}} > 4\\
    \Size{x_{2}} > \Size{x_{3}} & \text{if $m > 2$}\\
  \end{cases}
\\&
  \begin{cases}
    n = 1, m \ge 1\\
    \Size{x_{1}} = \Size{x_{2}} = \dots = \Size{x_{k}} \ge 4,
    &
    \text{for some $k \le m$}\\
    \Size{x_{k}} > \Size{x_{k+1}} & \text{if $k < m$}\\
  \end{cases}
\\&
  \begin{cases}
    n = 1, m = 1\\
    \Size{x_{1}} = 2\\
  \end{cases}
\\&
  \begin{cases}
    n = 2, m \ge 1\\
    \Size{x_{1}} > \Size{x_{2}} & \text{if $m > 1$}\\
  \end{cases}
  \end{align*}

\item 
    for all other groups $G$ we have
  $\Size{\Hc(G)} = \Size{T(G)} = 1$.
  \end{enumerate}
\end{theorem}

\section{The group $T(G)$}\
\label{sec:group}

We first record the following
\begin{lemma}\label{lemma:conjugation}
  In the notation of Section~\ref{sec:regular}, suppose $\theta \in
  S(G)$ is an isomorphism  $\theta : G
  \to (G, \circ)$.
  
  Then $\theta$ conjugates $\rho(G)$ to $N$.
\end{lemma}
\begin{proof}
  For $g, h \in G$ we have
  \begin{equation*}
    g^{\rho(h)^{\theta}}
    =
    g^{\theta^{-1} \rho(h) \theta}
    =
    (g^{\theta^{-1}} + h)^{\theta}
    =
    g \circ h^{\theta}
    =
    g^{\nu(h^{\theta})},
  \end{equation*}
  whence $\rho(h)^{\theta} = \nu(h^{\theta})$.
\end{proof}

In  the previous  section we  have  determined, for  a given  finitely
generated abelian group $G$, all regular subgroups $N$ of $S(G)$ which
are normal in $\Hol(G)$, that is, the elements of the set
\begin{equation*}
  \Kc(G)
  =
  \Set{ N \le S(G) : \text{$N$ is regular, $N \norm \Hol(G)$} }.
\end{equation*}
We weeded
out those $N \in \Kc(G)$ for which $N_{S(G)}(N) > N_{S(G)}(\rho(G))$,
  and seen that the remaining groups are isomorphic to $G$. Now if $N
\in \Kc(G)$ and $\theta : G \to (G, \circ) \cong N$ is an isomorphism, by
Lemma~\ref{lemma:conjugation} we have
\begin{equation}\label{eq:ineq}
  N_{S(G)}(\rho(G))^{\theta}
  =
  N_{S(G)}(\rho(G)^{\theta})
  =
  N_{S(G)}(N)
  \ge
  N_{S(G)}(\rho(G)).
\end{equation}
In this section we will prove the following Lemma.
\begin{lemma}\label{lemma:iso-order-two}
  For  each of  the  regular subgroups  $N  \cong G$  of the  previous
  section, there is $\theta \in  S(G)$ \emph{of order two} which is an
  isomorphism $\theta : G \to (G, \circ)$.
\end{lemma}
We will then have
from~\eqref{eq:ineq} 
\begin{equation*}
  N_{S(G)}(\rho(G))
  =
  N_{S(G)}(\rho(G))^{\theta^{2}}
  \ge
  N_{S(G)}(\rho(G))^{\theta},
\end{equation*}
so that
\begin{equation*}
  N_{S(G)}(N)
  =
  N_{S(G)}(\rho(G)).
\end{equation*}
Therefore the regular subgroups $N \cong G$  of the previous section
will turn out to be exactly the elements of the set
\begin{equation*}
  \Hc(G)
  =
  \Set{ N \le S(G) : \text{$N$ is regular, $N \cong G$ and $N_{S(G)}(N)
    = \Hol(G)$} }.
\end{equation*}

In the previous section we have shown that for each group structure
$(G, +)$ there are $1, 2$, or $4$ rings $(G, +, \cdot)$. Therefore we
will have obtained
\begin{theorem}\label{thm:TG}
  For each finitely generated abelian group $G$, the group $T(G)$ is
  elementary abelian, of order $1, 2$, or $4$
\end{theorem}

\begin{proof}[Proof of Lemma~\ref{lemma:iso-order-two}]
We now describe, for each of  the regular subgroups $N \cong G$ of the
previous section,  an element of $\theta \in S(G)$  of order two
which yields an 
isomorphism $\theta : G \to (G, \circ)$.

In the previous section we have
noted that in all  the cases of Theorem~\ref{thm:main}, the generators
$z_{i}, x_{j}$ are still generators of $(G, \circ)$, they retain their
orders in $(G, \circ)$, and $(G,  \circ)$ is still a direct product of
the   cyclic  subgroups   generated   by  the   $z_{i},  x_{j}$   (see
Remark~\ref{rem:is-iso}).  Therefore there is an isomorphism $\theta :
G  \mapsto  (G,  \circ)$   such  that  $z_{i}^{\theta}  =  z_{i}$  and
$x_{j}^{\theta} =  x_{j}$ for  all $j$.  This  can be extended  to the
whole of $G$ via
\begin{equation}\label{eq:fromGtoGcirc}
  (x + y)^{\theta} 
  =
  x^{\theta} \circ y^{\theta}
  =
  x^{\theta} + y^{\theta} + x^{\theta} y^{\theta}
\end{equation}
for all $x, y  \in G$.   

Define a
function
\begin{align*}
f :\ &G \to G\\
     &u \mapsto u^{\theta} - u.
\end{align*}
We will be using several times the following simple observation
\begin{equation}\label{lemma:f}
  f(G) \subseteq G^{2}.
\end{equation}

Recall from Lemma~\ref{lemma:threefold} that $2 G^{2} = 0$, and from
Theorem~\ref{thm:normal-regular}\eqref{item:ghk} that
$G^{2}$ lies in the
annihilator of the ring.

  To prove~\eqref{lemma:f},  we proceed by induction on  the length of
  $u$  as   a  sum   of  the  generators   $z_{i},  x_{j}$.   We  have
  from~\eqref{eq:fromGtoGcirc}, if $y$ is one of these generators,
  \begin{align*}
    f(u + y)
    &=
    (u + y)^{\theta}
    -
    (u + y)
    \\&=
    u^{\theta} - u
    +
    y^{\theta} - y
    +
    u^{\theta} y^{\theta}
    \\&=
    f(u) + u^{\theta} y^{\theta}
    \in G^{2},
  \end{align*}
  as $y^{\theta} = y$, for $y$ a generator.

We have thus proved~\eqref{lemma:f}.

Note that for all $u, v \in G$ we have
\begin{align*}
  (u + v)^{\theta}
  &=
  u^{\theta} + v^{\theta} + u^{\theta} v^{\theta}
  \\&=
  u + f(u) + v + f(v) + (u + f(u)) (v + f(v))
  \\&=
  u + v + f(u) + f(v) + u v,
\end{align*}
so that 
\begin{equation*}
  f(u + v) =  f(u) + f(v) + u v.
\end{equation*}
Therefore~\eqref{lemma:f}
yields $f(2 u) = 2 f(u) + u^{2} = u^{2}$, so that
\begin{equation}\label{eq:4}
  f(4 u) 
  =
  f(2 u + 2 u)
  =
  2 f(2 u) + 4 u^{2}
  =
  0.
\end{equation}
Therefore
\begin{equation}\label{eq:isaninvolution}
  \begin{aligned}
    u^{\theta^{2}}
    &=
    (u + f(u))^{\theta}
    \\&=
    u + f(u) + f(u + f(u))
    \\&=
    u + f(u) + f(u) + f(f(u)) + u f(u)
    \\&=
    u + f(f(u)),
  \end{aligned}
\end{equation}
by~\eqref{lemma:f}.

In  the cases of  Theorem~\ref{thm:main} when  $\Size{x_{1}} >  4$, we
have $f(G) \subseteq G^{2} \le 4 H$. Now~\eqref{eq:isaninvolution}~and
\eqref{eq:4} yield $u^{\theta^{2}} = u$.

In the cases when $\Size{x_{1}} = 4$, we have $G^{2} = \Span{t_{1}, \dots,
  t_{k}}$ for some $k$, and $x_{i}^{2} = 0$ for all $i$. Thus we have,
for $i \le k$, 
\begin{equation}\label{eq:vanishes}
  f(t_{i}) = f(2 x_{i}) = 2 f(x_{i}) + x_{i}^{2} = 0,
\end{equation}
so
that~\eqref{eq:isaninvolution}  implies 
$u^{\theta^{2}} = u$.

Finally, when $\Size{x_{1}} = 2$ in \eqref{eq:tf-nis1}, we have $f(t_{1}) =
f(x_{1}) = x_{1}^{\theta} - x_{1} = 0$.

Therefore $\theta \in S(G)$ is in all cases an involution, as claimed.
\end{proof}

\section*{Acknowledgements}

We  wish  to  extend  our  heartfelt gratitude  to  the  referee  for
providing a  number of very  insightful comments,  that have led  to a
distinct improvement of our manuscript.

%\section*{References}

%\bibliography{Refs}
\providecommand{\bysame}{\leavevmode\hbox to3em{\hrulefill}\thinspace}
\providecommand{\MR}{\relax\ifhmode\unskip\space\fi MR }
\providecommand{\MRhref}[2]{%
  \href{http://www.ams.org/mathscinet-getitem?mr=#1}{#2}
}
\providecommand{\href}[2]{#2}

\end{document}